\documentclass[11pt]{article}
\usepackage[scale=.675,centering]{geometry}
\usepackage{amssymb,amsmath}
\usepackage{amsthm}
\usepackage{latexsym}
\usepackage[round]{natbib}
\usepackage{color}
\usepackage{fleqn,latexsym}
\usepackage{color}
\newtheorem{theorem}{Theorem}[section]
\newtheorem{lemma}[theorem]{Lemma}

\newtheorem{corollary}[theorem]{Corollary}

\newtheorem{definition}[theorem]{Definition}

\newtheorem{example}[theorem]{Example}

\newtheorem{claim}{Claim}

\DeclareMathOperator{\RN}{RN}

\newcommand{\uhr}{\upharpoonright}
\renewcommand{\phi}{\varphi}

\title{The complexity of the set of validities of a theory
}

\author{
Denis R. Hirschfeldt\thanks{Partially supported by NSF grant DMS-1854279.} \\ University of Chicago \and
Henry Towsner\thanks{Supported by NSF grant DMS-2054379.} \\ University of Pennsylvania \and
Scott Weinstein \\ University of Pennsylvania}

\begin{document}
\maketitle

\section{Introduction}\label{intro-sec}

In this paper, we study the collection of first-order logical schemata all of whose instances are theorems of a given theory $T$; we call these the \textit{validities of} $T$ ($\mathsf{V}(T)$). It is easy to see (Corollary \ref{vt-cor} below) that if $T$ is a decidable theory, then $\mathsf{V}(T)$ is distinct from the set of valid formulas of first-order logic as customarily understood. We provide a complete model-theoretic characterization of the complexity, in the sense of Turing degree, of $\mathsf{V}(T)$ for decidable theories $T$; answer a question posed by
\cite{Vau60} concerning the complexity of the collection of validities
common to all decidable theories; and begin the study of the case of undecidable theories.

\section{The validities of a theory}\label{voft-sec}

We fix a first-order language $S$ with effective syntax, without identity, without constant or function symbols, and with countably many relation symbols $X^n_i$ of arity $n$, for each $n\geq 0$. We think of $S$ as a schematic language for formulating validities. We use  $L$ to denote an effective first-order language with vocabulary disjoint from $S$ that may contain identity, function symbols, and constants.\footnote{In somewhat dated terms, $S$ corresponds to the language of the pure predicate calculus without identity, and the languages $L$  to applied predicate calculi with identity. See, for example, \cite{kleene1952introduction}.} 
 \begin{definition}
    Let $\varphi(X^1_1, \ldots,X^1_k,\ldots,X^n_1, \ldots,X^n_k)$ be a formula of the first-order language $S$ all of whose predicate symbols are among those exhibited, and let 
    \[
    \psi^1_1(y_1), \ldots,\psi^1_k(y_1),\ldots,\psi^n_1(y_1,\ldots,y_n), \ldots,\psi^n_k(y_1,\ldots,y_n)
    \]
    be formulas of a first-order language $L$ in which the displayed variables $y_i$ all appear free. Then 
    \[
    \varphi(\psi^1_1(y_1), \ldots,\psi^1_k(y_1),\ldots,\psi^n_1(y_1,\ldots,y_n), \ldots,\psi^n_k(y_1,\ldots,y_n))
    \]
is the formula of $L$ obtained by replacing $X^j_i(t_{i1}\cdots t_{ij})$ with $\psi^j_i(t_{i1},\ldots,t_{ij})$ in $\varphi$, for $1\leq i \leq k$ and $1\leq j \leq n$.
  \end{definition}

We write $\mathsf{V}$ for the set of valid formulas of $S$. The following well-known theorem summarizes two of the fundamental results of mathematical logic, the G\"{o}del Completeness Theorem and the Church-Turing Undecidability Theorem.
 \begin{theorem}\label{gct-thm}
    $\mathsf{V}$ is a $\Sigma_1$-complete set.
  \end{theorem}
  
  The next definition introduces the notion of the ``set of validities of a theory,'' the central concept of this paper. 
  
 \begin{definition}\label{voft-def}
 A set of sentences $T$ of a first-order language $L$ is an
 \emph{$L$-theory} if and only if $T$ is closed under logical
 consequence. We assume all theories we consider are consistent.
 
 Let $T$ be an $L$-theory and let
 $\varphi(X^1_1, \ldots,X^1_k,\ldots,X^n_1, \ldots,X^n_k)$ be a sentence of the first-order language $S$ all of whose predicate symbols are among those exhibited. The sentence $\varphi$ is a validity of $T$ if and only if for all $L$-formulas
 \[
    \psi^1_1(y_1), \ldots,\psi^1_k(y_1),\ldots,\psi^n_1(y_1,\ldots,y_n), \ldots,\psi^n_k(y_1,\ldots,y_n)
    \]
    in which the displayed variables $y_i$ all appear free,
    \[
    \varphi(\psi^1_1(y_1), \ldots,\psi^1_k(y_1),\ldots,\psi^n_1(y_1,\ldots,y_n), \ldots,\psi^n_k(y_1,\ldots,y_n))
    \]
is a member of $T$.

We write $\mathsf{V}(T)$ for the set of valid sentences of $T$.   
   \end{definition}
   
   It is evident that for every theory $T$, $\mathsf{V}\subseteq\mathsf{V}(T).$ It is well-known that identity holds for some computably axiomatizable theories. The following result is a corollary to the ``Hilbert-Bernays Completeness Theorem.''\footnote{\textit{Cf.} Theorem 36 of \cite{kleene1952introduction}. The theory $\mathbf{PA}$ is first-order Peano Arithmetic. A perspicuous, new proof of this result is presented by \cite{ebbs}.} 
 \begin{theorem}\label{hb-comp-thm}
$\mathsf{V}(\mathbf{PA})=\mathsf{V}$.
\end{theorem}  
\section{The validities of decidable theories}\label{dec-sec}
In this section we focus on the complexity of the collection of validities of a decidable theory. 
\begin{theorem}\label{dec-val-pi1-thm}
If $T$ is a \textit{decidable}  theory, then  $\mathsf{V}(T)$
is $\Pi_1$. 
\end{theorem}

  \begin{proof}
    Let $\varphi(X_1,\ldots,X_k)$ be an $S$-sentence and suppose that $T$ is a decidable $L$-theory.  We can effectively enumerate all possible choices of $L$-formulas $\psi_1,\ldots,\psi_k$.  By Definition \ref{voft-def}
\[
\varphi(X_1,\ldots,X_k)\in\mathsf{V}(T)
\]
if and only if for all $L$-formulas $\psi_1,\ldots,\psi_k$,
\[
\varphi(\psi_1,\ldots,\psi_k)\in T.
\]
  \end{proof}
In contrast to Theorem \ref{hb-comp-thm}, we have the following result concerning decidable theories.
\begin{corollary}\label{vt-cor}
If $T$ is decidable, $\mathsf{V}\subsetneq\mathsf{V}(T)$. 
\end{corollary}
  \begin{proof}
   The result follows at once from    Theorems \ref{gct-thm} and \ref{dec-val-pi1-thm}.  
   \end{proof}
\begin{example}\label{dlo-ex}
    Consider $T_{dlo}=\mathsf{Th}(\mathbb{Q},<)$, the theory of dense linear orders without endpoints.
Let $\varphi(X)$ be the formula
\[\forall v,w(Xv\leftrightarrow Xw).\]
Then $\varphi(X)\in\mathsf{V}(T_{dlo})$, because for any formula $\psi(y)$ with a single free-variable in the language $\{<,=\}$, either $\forall y\,\psi\in T_{dlo}$
or
$\forall y\,\neg\psi\in T_{dlo}$. 
  \end{example}
   
In the remainder of this section, we explore the interaction between the complexity of $\mathsf{V}(T)$ and $\aleph_0$-categoricity. In particular, in Section \ref{omega-cat-subsec} we show, as a corollary to a result of \cite{Schmerl}, that for every c.e.\ degree $\mathbf{a}\in[\mathbf{0}, \mathbf{0^\prime}]$, there is a decidable $\aleph_0$-categorical theory $T$ with $\mathsf{deg}(\mathsf{V}(T))=\mathbf{a}$. In Section \ref{stab-sec}, we establish that if a decidable theory $T$ is not $\aleph_0$-categorical, then  $\mathsf{V}(T)$ is $\Pi_1$-complete.
\subsection{$\aleph_0$-categorical theories}\label{omega-cat-subsec}

We begin by showing that the set of  valid sentences of the theory of dense linear orders without endpoints is decidable.
 \begin{theorem}\label{dlo-rec-thm}
    $\mathsf{V}(T_{dlo})$ is decidable.
  \end{theorem}
  \begin{proof}
$T_{dlo}$ has quantifier elimination, so every formula on $n$ variables is equivalent to a disjunction of a conjunction of clauses $x_i<x_j$ or $x_i=x_j$.  There are at most $2^{3^{n^2}}$ such formulas.
For instance, the only formulas on $2$ variables are the $8$ Boolean combinations of
\[x_1<x_2,\ x_2<x_1,\ \text{and}\ x_1=x_2.\]
So to test the validity of $\varphi(X_1,\ldots,X_k)$ we can iterate through the finitely many choices for each $X_i$.
  \end{proof}

\begin{example}\label{dlo2-ex}
The proof of Theorem \ref{dlo-rec-thm} indicates further validities 
in $\mathsf{V}(T_{dlo})\setminus \mathsf{V}$ along the lines of Example \ref{dlo-ex}. For example, the following sentence is in $\mathsf{V}(T_{dlo})\setminus \mathsf{V}$.
\[
\bigvee_{1\leq i<j\leq 9}(\forall y_1)(\forall y_2)(X_i(y_1,y_2)\leftrightarrow X_j(y_1,y_2)).
\]
\end{example}
Theorem \ref{dlo-rec-thm} is a special case of a general result
characterizing the complexity of $\mathsf{V}(T)$ for all decidable
$\aleph_0$-categorical  theories $T$ in terms of an important
model-theoretic invariant associated to $T$.

\begin{theorem}[Engeler/Ryll-Nardzewski/Svenonius]
    A complete theory $T$ with infinite models is $\aleph_0$-cat\-e\-go\-ri\-cal if and only if for each $n$ there are finitely many $n$-types consistent with $T$.    
  \end{theorem}
  \begin{definition}
    The \emph{Ryll-Nardzewski function} of a complete $\aleph_0$-categorical theory $T$ is the function $\RN_T:\mathbb{N}\rightarrow\mathbb{N}$ such that, for each $n$, $\RN_T(n)$ is the number of $n$-types consistent with $T$.
  \end{definition}
  \begin{theorem}
    If $T$ is a complete, decidable, $\aleph_0$-categorical theory, then $\mathsf{V}(T)\equiv_\textup{T} \RN_T$.
  \end{theorem}
  \begin{proof}
  Let $X^n_i$, $1\leq i\leq 2^k$ be predicate letters of arity $n$ and let $\overline{y} = (y_1,\ldots,y_{n})$. Let 
    $\varphi(X_1,\ldots,X_{2^k})$ be the formula
    \[
\bigvee_{1\leq i<j\leq 2^k}(\forall \overline{y})(X_i(\overline{y})\leftrightarrow X_j(\overline{y})).
\]
    Then, $\varphi\in \mathsf{V}(t)$
    if and only if $\RN_T(n)<k$.
Therefore $\RN_T\leq_\textup{T} \mathsf{V}(T)$.

Given an oracle for $\RN_T$ and a decidable theory $T$, consider a formula $\varphi(X_1,\ldots,X_k)$.
For each $X_i$ with arity $n_i$ and $d_i=\RN_T(n_i)$, enumerate lists $(\psi_1,\ldots,\psi_{d_i})$ of formulas of arity $n_i$ until we find a list of formulas, each consistent with $T$, and pairwise nonequivalent modulo $T$.  Such a list must be an enumeration of the $n_i$-types of $T$.
Then all formulas of arity $n_i$ are equivalent to disjunctions of the $\psi_i$.  We can test $\varphi(X_1,\ldots,X_k)$ with each of the finitely many possible combinations of  these disjunctions.  Therefore $\mathsf{V}(T)\leq_\textup{T} \RN_T$.
\end{proof}
It follows at once that If $T$ is complete, decidable, and $\aleph_0$-categorical, and $\RN_T$ is computable, then $\mathsf{V}(T)$ is computable.\footnote{Notice that the same argument shows that if $T$ is a complete decidable theory with a finite model, then $\mathsf{V}(T)$ is computable.}

The following result of \cite{Schmerl} now yields a complete characterization of the possible Turing degrees of $\mathsf{V}(T)$ for decidable $\aleph_0$-categorical theories $T$.
 \begin{theorem}[Schmerl]
     For any c.e.\ degree $\mathbf{a}\in[0,0']$, there are complete, decidable, $\aleph_0$-categorical theories $T$ such that $\RN_T$ is in $\mathbf{a}$.
   \end{theorem}
  \begin{corollary}\label{schmerl-cor}
    For any c.e.\ degree $\mathbf{a}\in[0,0']$, there are complete, decidable, $\aleph_0$-categorical theories $T$ such that $\mathsf{V}(T)$ is in $\mathbf{a}$.
  \end{corollary}
  Note that Corollary \ref{schmerl-cor} is optimal, since $\mathsf{V}(T)$ must have c.e.\ degree, by Theorem \ref{dec-val-pi1-thm}.
  
  As a corollary to the foregoing proof that $\RN_T\leq_\textup{T} \mathsf{V}(T)$, we have
   \begin{corollary}\label{vrn-cor}
    If $T$ and $U$ are complete $\aleph_0$-categorical theories, then
  \[
  \text{if }\mathsf{V}(T)=\mathsf{V}(U),\text{then } \RN_T=\RN_U.
  \]
\end{corollary}
The following example, due to Alex Kruckman, shows that the converse to Corollary \ref{vrn-cor} fails.\footnote{Alex Kruckman, private communication.}
\begin{example}
There are complete $\aleph_0$-categorical theories $T$ and $U$ such that
\[
\RN_T=\RN_U,\text{ but }\mathrm{V}(T)\neq\mathrm{V}(U).
\]
For example, let $T$ be the theory of the countable random graph, and let $U$ be the theory of the countable random tournament.
Since both theories admit elimination of quantifiers, $k$-types are determined, in the case of $T$, by the presence or absence of edges, and in the case of $U$, by the orientation of edges. Hence, $\RN_T=\RN_U$.
On the other hand, for all distinct $a,b,c,d\in |U|$,
\[
\mathsf{type}_U(a,b)=\mathsf{type}_U(c,d)\text{ or } \mathsf{type}_U(a,b)=\mathsf{type}_U(d,c),
\]
which generates a formula in $\mathrm{V}(U)$ that is absent from $\mathrm{V}(T)$.
\end{example}

\subsection{Non-$\aleph_0$-categorical theories}\label{stab-sec}

\begin{theorem}\label{dec-non-a0cat-thm}
If $T$ is a complete, decidable, non-$\aleph_0$-categorical theory with infinite models, then $\mathsf{V}(T)$ is $\Pi_1$-complete.
\end{theorem}
\begin{proof}
Let $T$ be a complete, decidable, non-$\aleph_0$-categorical theory with infinite models. We will show that 
\begin{claim}\label{main-eq}
there is a pair of effectively inseparable c.e.\ sets $Z^0$ and $Z^1$
 and a set $Y$ such that 
 \begin{enumerate}
 \item
 $Y$ separates $Z^0$ and $Z^1$, and
 \item
$Y$ is computable in  $\mathsf{V}(T)$.
\end{enumerate}
\end{claim}
It follows from Claim (\ref{main-eq}) that  $\mathsf{V}(T)$ has PA degree, and from Theorem \ref{dec-val-pi1-thm}  that it has c.e.\ degree. Since the only c.e.\ PA degree is $0^\prime$, we conclude that  $\mathsf{V}(T)$ is $\Pi_1$-complete.\footnote{The reader is advised to consult \cite[\S2.21]{DowneyHirschfeldt} for background on PA degrees.}
We proceed to establish Claim (\ref{main-eq}).

We consider Turing machines with the unary tape alphabet $\{1, B\}$. We think of such a machine as computing 0-1-valued partial functions on $\omega$ via the convention that input $n$ is represented by a tape initialized with 1's on its  $n+1$ left-most squares and otherwise blank, and with output 0, if it halts on a blank square, and with output 1, if it halts on a square inscribed with a 1. Let $M$ be such a Turing machine with the property that the sets 
\begin{align}
Z^0 &= \{n\mid M\mbox{ outputs 0 with input $n$}\}, \mbox{ and}
\\
Z^1 &= \{n\mid M\mbox{ outputs 1 with input $n$}\},
\end{align}
are effectively inseparable. 
In order to establish Claim \ref{main-eq}, it suffices to show that there is a total 0-1-valued function $f$ computable in  $\mathsf{V}(T)$ that extends the partial function computed by $M$. Toward this end, we specify formulas $\varphi^l_n$, $n\in \omega$, $l\in\{0,1\}$ that describe the behavior of $M$ with input $n$. We suppose that $M$ has $m$ states labeled $1,\dots,m$.

Since $T$ is not $\aleph_0$-categorical, we may choose a $k$ so that there are infinitely many $k$-types. 
For all $n\in \omega$, $l\in\{0,1\}$, the vocabulary of $\varphi^l_n$ consists of the following predicates:
\begin{itemize}
\item $X_{\mathrm{step}}, X_{\mathrm{start}},X_{\mathrm{end}},X_{\mathrm{tape}}, X_{\mathrm{left}},$ and  $X_{\mathrm{state}=i}, i\in\{1,\ldots,m\}$ are $k$-ary predicates, and
\item $X_\leq, X_{\preceq},X_0,X_1,
X_{\mathrm{head}}$ 
are all $2k$-ary predicates.
\end{itemize}
The sentences $\varphi_n^l$ will state that the predicates encode the execution of the Turing machine $M$ with input $n$ and output $l$, in accord with the foregoing input-output conventions.  
$X_{\mathrm{step}}$ will be a predicate satisfied by those elements that encode a time step in the running of $M$.  $X_{\leq}$ will be the subset of $X_{\mathrm{step}}^2$ encoding when one step in time happens before another.  In general we cannot expect that a time step is represented by a single element, so we will only require that $X_{\leq}$ be a preorder on $X_{\mathrm{step}}$.  $X_{\mathrm{start}}$ and $X_{\mathrm{end}}$ encode the first and last time steps, respectively. 
Similarly, $X_{\mathrm{tape}}$ will be a predicate satisfied by those elements that encode a position on the tape.  $X_{\preceq}$ will encode a preorder on $X_{\mathrm{tape}}$ and $X_{\mathrm{left}}$ will be the leftmost position.

$X_0$ and
$X_1$ 
partition $X_{\mathrm{step}}\times X_{\mathrm{tape}}$ encoding which positions on the tape are blank ($X_0$) or hold a $1$ ($X_1$)
at a given time step, and $X_{\mathrm{head}}$ will be a subset of $X_{\mathrm{step}}\times X_{\mathrm{tape}}$ encoding where the head is at a given step.  

Finally, the sets $X_{\mathrm{state}=1},\ldots,X_{\mathrm{state}=m}$ will be subsets of $X_{\mathrm{step}}$ indicating what state the machine is in at a given step.

Formally, $\varphi^l_n$ is a conjunction of clauses stating:
\begin{itemize}
\item $X_\leq$ is a discrete preorder on $X_{\mathrm{step}}$ with $X_{\mathrm{start}}$ and $X_{\mathrm{end}}$ as the first and last equivalence classes, respectively;
\item $X_\preceq$ is a discrete preorder on $X_{\mathrm{tape}}$ with $X_{\mathrm{left}}$ as the first equivalence class;
\item $X_0$ and $X_1$ 
partition $X_{\mathrm{step}}\times X_{\mathrm{tape}}$ respecting $X_\leq$ and $X_\preceq$ equivalence classes; 
\item a clause stating that in the initial configuration the leftmost $n+1$ tape squares hold a $1$ and the remainder of the tape is blank;
\item $X_{\mathrm{head}}$ is a subset of $X_{\mathrm{step}}\times X_{\mathrm{tape}}$ that respects $X_\leq$ and $X_\preceq$ equivalence classes. 
$X_{\mathrm{head}}$ represents the graph of a function from $X_{\mathrm{step}}$ equivalence classes to $X_{\mathrm{tape}}$ equivalence classes, and contains  $X_{\mathrm{start}}\times X_{\mathrm{left}}$;
\item $X_{\mathrm{state}=1}$ through $X_{\mathrm{state}=m}$ form a partition of $X_{\mathrm{step}}$ respecting $X_\leq$ equivalence classes;
\item letting $U$ be the set of terminal states, $X_{\mathrm{end}}=\bigcup_{i\in U}X_{\mathrm{state}=i}$;
\item the $X_{\mathrm{head}}$ image of $X_{\mathrm{end}}$ is contained in $X_l$;
\item for each rule of the Turing machine saying ``when in state $i$ with $a$ under the head, change the head to $b$, change to state $j$, and move in a particular direction'', the implication corresponding to this rule.
\end{itemize}

We proceed to show that there is a total 0-1-valued function $f$ computable in  $\mathsf{V}(T)$ that extends the partial function computed by $M$.
First, suppose that $M$ halts on input $n$ with output $l\in\{0,1\}$. Then, for some $d>n$, $M$ halts with input $n$ in at most $d$ many steps with output $l$. We need at most $2d$ many elements to encode this computation as above: $d$ of them to represent the steps of the computation, and $d$ of them to represent the tape positions the head might occupy during the computation. By our choice of $k$,  we can specify a list $\psi_1,\ldots,\psi_{2d}$ of formulas on $k$-tuples, that are consistent with $T$ and pairwise incompatible over $T$.   We can then construct a substitution instance $\Phi^l_n$ of the formula $\varphi^l_n$ by substituting  all the predicate variables with suitable boolean combinations of these formulas so that $T\cup\{\Phi^l_n\}$ is satisfiable. It follows at once that for $l\in\{0,1\}$ and $n\in\omega$,
\begin{equation}\label{eq1}
\text{if $\neg\varphi^l_n\in \mathsf{V}(T)$, then $M$ does not halt on input $n$ with output $l$}.
\end{equation}
Moreover, for $l\in\{0,1\}$ and $n\in\omega$,
\begin{equation}\label{eq2}
\text{if $\neg\varphi^l_n\not\in \mathsf{V}(T)$, then $M$ does not halt on input $n$ with output $1-l$}.
\end{equation}
For suppose $\neg\varphi^l_n\not\in \mathsf{V}(T)$. It follows that $\varphi^l_n$ is satisfiable, and hence that there is a computation $C$ on input $n$ in which $M$ enters an end state scanning $l$.\footnote{\label{Cfoot}Note that $C$ may be a ``non-standard computation,'' that is, a linearly ordered sequence of steps which behaves like a computation, but need not have finite length.}  But then if $M$ halted on input $n$ with output $1-l$, $C$ would contain a finite initial segment in which $M$ enters an end state scanning $1-l$. Thus, it follows that $M$ does not halt on input $n$ with output $1-l$.

It is evident from (\ref{eq1}) and (\ref{eq2}) that the following instructions specify a total 0-1-valued function $f$ computable in  $\mathsf{V}(T)$ that extends the partial function computed by $M$.
Given input $n$, output $l$ for the least $l\in\{0,1\}$ such that $\neg\varphi^l_n\not\in \mathsf{V}(T)$; if there is no such $l$, output $0$. 
\end{proof}

\section{Validities of all decidable theories
}\label{all-dec-subsec}

In this section we follow \cite{Vau60} and study the validities common to all decidable theories. We begin with the following definition from \cite{Vau60}.\footnote{\cite{Vau60} remarks that $\mathsf{V_{dec}}$ was considered by Tarski, though he gives no citation, and by \cite{cobham}.
}
  \begin{definition}
    $\mathsf{V_{dec}}=\bigcap_T\mathsf{V}(T)$ where $T$ ranges over all decidable theories.
    
  \end{definition}

The next two results concerning the complexity of  $\mathsf{V_{dec}}$  are due to \cite{cobham} and \cite{Han65}, respectively.
 \begin{theorem}[Cobham]\label{cobham-thm}
    $\mathsf{V_{dec}}$ is $\Pi_3$.
  \end{theorem}
  \begin{theorem}[Hanf]\label{hanf-thm}
    $\mathsf{V_{dec}}$ is not $\Pi_1$.
  \end{theorem}
  
  \cite{Vau60} framed the problem of classifying $\mathsf{V_{dec}}$  
  in the arithmetical hierarchy; 
  indeed, Theorem \ref{hanf-thm} was the outcome of Hanf's successful effort to resolve Vaught's specific question whether $\mathsf{V_{dec}}$ is co-c.e.

  The next theorem resolves  
  Vaught's problem.
   \begin{theorem}\label{vidpi3-thm}
      $\mathsf{V_{dec}}$ is $\Pi_3$-complete.
  \end{theorem}

The proof of Theorem \ref{vidpi3-thm} makes use of the following results due to \cite{Han65},  \cite{Han75} and \cite{fat},\footnote{Specifically,  Theorem \ref{hanf2-thm} is a corollary to Peretyat'kin's result that  from every computably axiomatizable theory $A$, a finitely axiomatizable theory $F$ can be constructed such that $F$ is computably isomorphic to $A$; moreover the construction of $F$ is effective in a c.e.\ index for $A$. This result was conjectured in \cite{Han65} with a proof sketched in \cite{Han75}. A full proof is presented in \cite{fat}. } 
and to \cite{Gasarch1993}, respectively. We write $T_e$ for the $\Pi^0_1$ subclass of $2^\omega$ with index $e$, and $\theta_e$ for an effective enumeration of the sentences of the language $L$ of graphs, which contains a single binary predicate $E$ and the identity predicate.
\begin{theorem}[Hanf/Peretyat'kin]\label{hanf2-thm}
There is a computable function $\sigma$ such that, for each $e$, 
the degrees of paths in the $\Pi^0_1$-class $T_e$ is the same as the degrees of complete $L$-extensions of the theory axiomatized by the sentence $\theta_{\sigma(e)}$; in particular, $T_e$ has a computable path if and only if $\theta_{\sigma(e)}$ has a decidable complete $L$-extension. \end{theorem}
\begin{theorem}[Gasarch-Martin]\label{gasarch-thm}
The set of $e$ such that $T_e$ has no computable path is $\Pi_3$-complete.
\end{theorem}

Since our schematic language $S$ for validities does not contain identity, we need to specify a translation of $L$-sentences $\theta$ into $S$-sentences $\theta^*$ in order to derive  Theorem \ref{vidpi3-thm} from Theorems \ref{hanf2-thm} and \ref{gasarch-thm}. Let $R$ and $\approx$ be binary predicate letters of $S$ and let $\chi(R,\approx)$ be the $S$-sentence that expresses that $\approx$ is a congruence with respect to $R$. If $\theta$ is an $L$-sentence, $\theta^*$ is the $S$-sentence 
\[\chi(R,\approx)\wedge \theta[E/R,{=}/{\approx}].\]
Theorem \ref{vidpi3-thm} is an immediate corollary to Theorems \ref{hanf2-thm} and  \ref{gasarch-thm}, and the next lemma.
\begin{lemma}\label{vidpi3-lem}
An $L$-sentence $\theta$  has a complete decidable $L$-extension if and only if $\neg\theta^*\not\in\mathsf{V_{dec}}$.
\end{lemma}
 \begin{proof}
 The left to right direction is immediate: if $T$ is a decidable complete $L$-extension of $\theta$, then $\neg\theta^*\not\in\mathsf{V}(T)$, hence $\neg\theta^*\not\in\mathsf{V_{dec}}$.
 
 For the right to left direction, suppose that  $\neg\theta^*\not\in\mathsf{V_{dec}}$, and let $T$ be a decidable theory with $\neg\theta^*\not\in\mathsf{V}(T)$. Thus, there are formulas $\varphi$ and $\psi$ in the language of $T$ such that 
 \begin{equation}\label{lem-eq-1}
 \neg\theta^*[R/\varphi,{\approx}/\psi]\not\in T.
 \end{equation}
 It follows from (\ref{lem-eq-1}) that we may extend $T$ to a decidable complete theory $T^*$ with 
  \begin{equation}\label{lem-eq-2}
\{\theta^*[R/\varphi,{\approx}/\psi],(\forall x)(\forall y)(Exy\leftrightarrow\varphi), (\forall x)(\forall y)(x\approx y\leftrightarrow\psi)\}\subseteq T^*.
 \end{equation}
But now it follows from (\ref{lem-eq-2}) that there is a complete decidable $L$-extension of $\theta$.
 \end{proof}
 \begin{proof}(of Theorem \ref{vidpi3-thm})
 It follows at once from Theorem \ref{hanf2-thm} and Lemma  \ref{vidpi3-lem} that the map from $e$ to $\neg\theta_{\sigma(e)}^*$ is a computable reduction from the collection of $\Pi^0_1$ classes $T_e$ with no computable path to $\mathsf{V_{dec}}$. Hence, by Theorems  \ref{cobham-thm} and \ref{gasarch-thm}, $\mathsf{V_{dec}}$ is $\Pi_3$-complete.
 \end{proof}

Note that the above proof still works if we restrict $T$ to decidable
complete theories in the definition of $\mathsf{V_{dec}}$.

Let $\mathsf{V_{fa}}=\bigcap_T\mathsf{V}(T)$ where $T$ ranges over all finitely axiomatizable complete theories.\footnote{\cite{szmielew} introduce the study of $\mathsf{V_{fa}}$.}
 Vaught notes that $\mathsf{V_{fa}}$ is $\Pi_3$, and suggests the problem of determining its exact arithmetical classification.\footnote{\cite{Vau60}, p.~51.} So far as we are aware, this problem remains open.

\section{The validities of undecidable theories}

It is also interesting to consider the complexity of $\mathsf{V}(T)$
when $T$ is undecidable. In the proof of
Theorem~\ref{dec-non-a0cat-thm}, the decidability of $T$ is not used
in showing that $\mathsf{V}(T)$ has PA degree, so we have the
following result.

\begin{theorem}\label{pathm}
If $T$ is a non-$\aleph_0$-categorical complete theory with infinite
models, then $\mathsf{V}(T)$ has PA degree.
\end{theorem}

Combining this fact with the extension of Arslanov's Completeness
Criterion by \cite*{JLSS}, we have the
following extension of Theorem~\ref{dec-non-a0cat-thm}.

\begin{corollary}
Let $\mathbf{d}$ be an $n$-CEA degree for some $n$. If $T$ is a
$\mathbf{d}$-computable, non-$\aleph_0$-categorical, complete theory
with infinite models, then $\mathbf{d} \oplus \mathsf{V}(T)$ computes
$\emptyset'$.
\end{corollary}

\begin{proof}
By Theorem~\ref{pathm}, the degree of $\mathsf{V}(T)$ is PA, and by
the relativized form of Theorem~\ref{dec-val-pi1-thm}, it is also
$\mathbf{d}$-c.e. Thus the degree of $\mathbf{d} \oplus \mathsf{V}(T)$
is PA and $(n+1)$-CEA. By the aforementioned work of Jockusch, Lerman,
Soare, and Solovay, this degree computes $\emptyset'$.
\end{proof}

Model-theoretic properties can also have an effect on the complexity
of $\mathsf{V}(T)$. For instance, as noted in Footnote~\ref{Cfoot}, in
the proof of Theorem~\ref{dec-non-a0cat-thm}, the computation $C$
might be non-standard. But in this case, still using the notation in
that proof, the formula substituted in for $X_\leq$ in witnessing the
satisfaction of $\phi^l_n$ defines a preorder with an infinite strict
chain, which is one way to say that $T$ has the strict order
property. Thus, if T is NSOP, i.e., if it does not have the strict
order property, then $\neg\phi^l_n \notin \mathsf{V}(T)$ actually
implies that $M$ halts on input $n$ with output $l$. This fact allows
us to compute the sets $Z^0$ and $Z^1$, rather than merely separate
them, using $\mathsf{V}(T)$. Choosing M to be a universal Turing
machine now yields the following.

\begin{theorem}
If $T$ is a non-$\aleph_0$-categorical complete NSOP theory with
infinite models, then $\mathsf{V}(T)$ computes $\emptyset'$.
\end{theorem}

The following result stands in contrast to the previous two results.

\begin{theorem}
Let $\mathbf{d}$ be a PA degree. Then there is a
$\mathbf{d}$-computable, complete, non-$\aleph_0$-categorical theory
$T$ with infinite models, such that $\mathsf{V}(T)$ is also
$\mathbf{d}$-computable.
\end{theorem}

\begin{proof}
We will use the fact that if $\Gamma$ is a consistent set of sentences
with no finite models in a language $\mathcal L_0$, and $\sigma$ is a
consistent, equality-free sentence in a language $\mathcal L_1$
disjoint from $\mathcal L_0$, then $\Gamma \cup \{\sigma\}$ is a
consistent set of $(\mathcal L_0 \cup \mathcal L_1)$-sentences.

Let $\mathcal L$ be a language with infinitely many predicate symbols
of each arity. (The proof will work the same whether or not $\mathcal
L$ includes equality.) Let $P_0,P_1,\ldots$ form an infinite,
coinfinite subset of the unary predicates of $\mathcal L$, and let
$\Gamma$ consist of $\exists x P_i(x)$ for each $i$, and $\forall x
\neg (P_i(x) \wedge P_j(x))$ for each $i \neq j$. Then $\Gamma$ is
consistent and has no finite models, and any consistent extension of
$\Gamma$ has infinitely many $1$-types, and hence cannot be
$\aleph_0$-categorical.

For each binary string $\alpha$, we will define a (not necessarily
consistent) set of $\mathcal L$-sentences $\Sigma_\alpha$, and a set
of $\mathcal S$-sentences $V_\alpha$. Let $\sigma_0,\sigma_1,\ldots$
be an effective list of all $\mathcal L$-sentences, and let
$\phi_0,\phi_1,\ldots$ be an effective list of all $\mathcal
S$-sentences (both without repetitions). For the empty string
$\lambda$, let $\Sigma_\lambda=\Gamma$ and let
$V_\lambda=\emptyset$. Given $\Sigma_\alpha$ and $V_\alpha$, proceed
as follows.

If $|\alpha|=2n$ for some $n$ then let $\Sigma_{\alpha 0}=\Sigma_\alpha
\cup \{\neg \sigma_n\}$, let $\Sigma_{\alpha 1}=\Sigma_\alpha \cup
\{\sigma_n\}$, and let $V_{\alpha 0} = V_{\alpha 1} = V_\alpha$.

If $|\alpha|=2n+1$ for some $n$ then let $V_{\alpha 0} = V_\alpha$,
let $V_{\alpha 1} = V_{\alpha} \cup \{\phi_n\}$, let $\Sigma_{\alpha
1} = \Sigma_\alpha$, and define $\Sigma_{\alpha 0}$ as
follows. Write $\phi_n$ as
$\phi_n(X^1_1,\ldots,X^1_k,\ldots,X_1^m,\ldots,X_k^m)$ (i.e., the
predicate symbols are among the ones exhibited). Let
$P^1_1,\ldots,P^1_k,\ldots,P_1^m,\ldots,P_k^m$ be pairwise distinct
predicate symbols of $\mathcal L$ such that each $P^j_i$ has arity
$j$, and no $P^j_i$ appears in $\Sigma_\alpha$. Let $\Sigma_{\alpha 0}
= \Sigma_\alpha \cup
\{\neg\phi_n(P^1_1,\ldots,P^1_k,\ldots,P_1^m,\ldots,P_k^m)\}$. Notice
that if $\Sigma_\alpha$ is consistent and $\phi_n$ is not a validity,
then $\Sigma_{\alpha 0}$ is consistent, by the fact mentioned in the
beginning of the proof.

Now for an infinite binary sequence Z, let $T_Z = \bigcup_{\alpha
\prec Z} \Sigma_{\alpha}$, and let $V_Z = \bigcup_{\alpha \prec Z}
V_\alpha$. It is easy to check that $T_Z$ and $V_Z$ are both
$Z$-computable. Now suppose that $T_Z$ is consistent. It is easy to
check that $T_Z$ is a complete, non-$\aleph_0$-categorical
theory. Furthermore, if $\phi_n \notin V_Z$, then $Z(2n+1)=0$, so
$T_Z$ contains a sentence witnessing that $\phi_n \notin
\mathsf{V}(T_Z)$. Thus $V_Z = \mathsf{V}(T_Z)$ if{}f $V_Z \subseteq
\mathsf{V}(T_Z)$, i.e., if{}f $\{\phi_n : Z(2n+1)=1\} \subseteq
\mathsf{V}(T_Z)$.

Let $\mathcal P$ be the class of all $Z$ such that $T_Z$ is consistent
and $\{\phi_n : Z(2n+1)=1\} \subseteq \mathsf{V}(T_Z)$. Both of these
conditions are $\Pi^0_1$, so $\mathcal P$ is a $\Pi^0_1$ class. We now
show that it is nonempty. Define a binary sequence $Y$ as
follows. Given $\alpha = Y \uhr 2n$, let $Y(2n)=1$ if $\Sigma_{\alpha}
\cup \{\sigma_n\}$ is consistent, and let $Y(2n)=0$ otherwise. Let
$Y(2n+1)=1$ if $\phi_n \in \mathsf{V}$ (i.e., $\phi_n$ is a validity),
and let $Y(2n+1)=0$ otherwise. It is easy to show by induction that
$\Sigma_\alpha$ is consistent for all $\alpha \prec Y$, so that $T_Y$
is consistent, and clearly $\{\phi_n : Y(2n+1)=1\} = \mathsf{V}
\subseteq \mathsf{V}(T_Y)$.

Thus $\mathcal P$ is a nonempty $\Pi^0_1$ class, and hence, by the
Scott Basis Theorem, has a $\mathbf{d}$-computable member $Z$. Then
$T_Z$ is a $\mathbf{d}$-computable, complete,
non-$\aleph_0$-categorical theory with infinite models, and
$\mathsf{V}(T_Z)=V_Z$ is also $\mathbf{d}$-computable.
\end{proof}

Notice that the sequence $Y$ in the above proof is
$\emptyset'$-computable, and $T_Y$ is thus a $\emptyset'$-computable
complete theory such that $\mathsf{V}(T_Y)=\mathsf{V}$. Conversely, if
$\mathsf{V}(T) = \mathsf{V}$ then $\mathsf{V}$ is both $\Sigma_1$ and
$\Pi_1$ relative to $T$, so $T$ computes $\mathsf{V}$, and hence
computes $\emptyset'$.

Notice also that, in the above proof, for any $\phi_n \notin
\mathsf{V}$, the restriction of $\mathcal P$ to elements $Z$ such that
$Z(2n+1)=0$ is still a $\Pi^0_1$ class containing $Y$, and hence has a
$\mathbf{d}$-computable member. Thus, for any PA degree $\mathbf{d}$,
the set of sentences that are validities of every
$\mathbf{d}$-computable complete theory is just $\mathsf{V}$, and
hence is $\Sigma^0_1$, in contrast to Theorem~\ref{vidpi3-thm}. It
would be interesting to study this set for degrees $\mathbf{d}$ that
are neither computable nor PA.

There is also likely more that can be said about the possible
complexity of $\mathsf{V}(T)$ for undecidable $T$, including how
this complexity is affected by model-theoretic properties of $T$
beyond $\aleph_0$-categoricity and NSOP.

\bibliography{HTW}

\end{document}